\documentclass[reqno,11pt,a4paper]{amsart}
\pdfoutput=1
\usepackage[ascii]{inputenc}
\usepackage[bookmarksnumbered,linktocpage,hypertexnames=false,colorlinks=true,linkcolor=blue,urlcolor=blue,citecolor=blue,anchorcolor=green,breaklinks=true,pdfusetitle]{hyperref}
\usepackage{bbm,braket,microtype,mathrsfs,amsmath,amssymb,xcolor,amsthm,amsfonts,latexsym,mathtools,graphicx,enumitem,booktabs,bm,xspace,float,mathdots,caption,subcaption,ellipsis,mleftright,amsrefs}
\usepackage[T1]{fontenc}
\usepackage{textcomp}
\usepackage[english]{babel}
\usepackage[capitalize,nameinlink]{cleveref}
\newtheorem{theorem}{Theorem}[section]
\newtheorem{lemma}[theorem]{Lemma}
\newtheorem{proposition}[theorem]{Proposition}
\newtheorem{corollary}[theorem]{Corollary}

\theoremstyle{definition}
\newtheorem{definition}[theorem]{Definition}

\newtheorem{example}[theorem]{Example}
\AddToHook{env/lemma/begin}{\crefalias{theorem}{lemma}}
\AddToHook{env/proposition/begin}{\crefalias{theorem}{proposition}}
\AddToHook{env/corollary/begin}{\crefalias{theorem}{corollary}}
\AddToHook{env/conjecture/begin}{\crefalias{theorem}{conjecture}}
\AddToHook{env/definition/begin}{\crefalias{theorem}{definition}}
\AddToHook{env/remark/begin}{\crefalias{theorem}{remark}}
\AddToHook{env/example/begin}{\crefalias{theorem}{example}}
\numberwithin{equation}{section}

\DeclareMathOperator{\GL}{GL}
\DeclareMathOperator{\Gr}{Gr}
\DeclareMathOperator{\Hom}{Hom}

\DeclareMathOperator{\Iso}{Iso}
\DeclareMathOperator{\Euler}{Euler}
\newcommand{\R}{\mathbbm R}
\newcommand{\C}{\mathbbm C}
\newcommand{\Z}{\mathbbm Z}

\newcommand{\op}{\oplus}

\allowdisplaybreaks[4]
\DeclareMathOperator{\Id}{Id}
\DeclareMathOperator{\geo}{geo}
\DeclareMathOperator{\rep}{rep}
\DeclareMathOperator{\Sub}{Sub}

\newcommand{\g}{{\mathfrak{g}}}
\newcommand{\z}{{\mathfrak{z}}}

\newcommand{\CE}{{\mathcal E}}
\newcommand{\CF}{{\mathcal F}}
\newcommand{\CG}{{\mathcal G}}
\newcommand{\CH}{{\mathcal H}}

\newcommand{\CP}{{\mathcal P}}

\newcommand{\CS}{{\mathcal S}}

\newcommand{\CV}{{\mathcal V}}
\newcommand{\CW}{{\mathcal W}}
\newcommand{\vn}{{\bm n}}
\newcommand{\vsigma}{{\bm  \sigma}}
\newcommand{\valpha}{{\bm \alpha}}
\newcommand{\vbeta}{{\bm \beta}}
\newcommand{\vtau}{{\bm \tau}}
\newcommand{\vdelta}{{\bm \delta}}
\newcommand{\vlambda}{{\bm \lambda}}

\newcommand{\vtheta}{{\bm  \theta}}

\newcommand{\vgamma}{{\bm \gamma}}
\newcommand{\vA}{{\bm A}}
\newcommand{\vB}{{\bm B}}

\newcommand{\vN}{{\bm N}}
\newcommand{\vxi}{{\bm  \xi}}
\begin{document}
\title[Quiver subrepresentations and the saturation property]{Quiver subrepresentations and the Derksen-Weyman saturation property}
\author{Velleda Baldoni}
\address{Velleda Baldoni: Dipartimento di Matematica, Universit\`a degli studi di Roma ``Tor Vergata'', Via della ricerca scientifica 1, I-00133, Italy}
\email{baldoni@mat.uniroma2.it}
\author{Mich\`ele Vergne}
\address{Mich\`ele Vergne: Universit\'e Paris Cit\'e, Institut Math\'ematique de Jussieu, Sophie Germain, case 75205, Paris Cedex 13, France}
\email{michele.vergne@imj-prg.fr}
\author{Michael Walter}
\address{Michael Walter: Ruhr University Bochum, Faculty of Computer Science, Universit\"atsstr.~150, 44801 Bochum, Germany}
\email{michael.walter@rub.de}
\date{}
\hypersetup{pdfauthor={Velleda Baldoni, Michèle Vergne, Michael Walter}}
%%=============================================================================

\begin{abstract}
Using Schofield's characterization of the dimension vectors of general subrepresentations of a representation of a quiver, we give a direct proof of the Derksen-Weyman saturation property.
\end{abstract}
%%=============================================================================
\maketitle
%%=============================================================================

%%=============================================================================
\section{Introduction}
%%=============================================================================
Let $Q=(Q_0,Q_1)$ be a quiver, let $\vn=(n_x)_{x\in Q_0}$ be a dimension vector, and $\GL_Q(\vn)=\prod_{x\in Q_0}GL(n_x)$.
Given a character $\det(g)^{\vsigma}$ of $\GL_Q(\vn)$ such that $\vsigma$ satisfy King's inequalities (\cref{def:Kingcone}), we construct a non-zero semi-invariant polynomial function of weight~$\vsigma$ (depending on an additional parameter) on the space $\CH_Q(\vn)$ of representations of the quiver~$Q$ in dimension~$\vn$.
Our proof directly uses Schofield's characterization of the dimension vectors of general subrepresentations of a representation of~$Q$.
This reproves the celebrated theorem of Derksen-Weyman~\cite{MR1758750} that the semi-group of weights of semi-invariants is saturated.
Particular consequences of this are the saturation property for the Littlewood-Richardson or Clebsch-Gordan coefficients of the general linear group, first proved by~Knutson-Tao~\cite{MR1671451} and as discussed in~\cites{MR1758750,Cra-Bo-Ge}, and more generally the saturation property for the multiplicity function for the action of $\GL_Q(\vn)$ on polynomials on $\CH_Q(\vn)$, as indicated to us by Ressayre~\cite{ressayrepc} (see Theorem~1.5 in~\cite{MR4671379}).

Derksen-Weyman deduced the saturation property from a main theorem, which states that the semi-invariants are linearly spanned by the so-called determinantal semi-invariants (\cref{def:cvw}).
The nonzero semi-invariants that we obtain in our proof are particular such semi-invariants.
However, we do not require (but also do not recover) the result that all semi-invariants are determinantal~\cites{MR1908144,MR1758750}.
Instead, our proof is based on the simple observation that King's inequalities for weights of semi-invariants are ``the same'' as Schofield's inequalities for subdimension vectors %(or quotient dimension vectors)~
(\cref{pro:king vs schofield}).

%=============================================================================
\section{Schofield Vectors}\label{sec:repQ}
%=============================================================================
Let $Q=(Q_0,Q_1)$ be a quiver without oriented cycles, where~$Q_0$ is the finite set of vertices and~$Q_1$ the finite set of arrows.
We use the notation~$a : x\to y$ for an arrow~$a\in Q_1$ from~$x\in Q_0$ to~$y\in Q_0$.
A \emph{dimension vector} for~$Q$ is a family~$\vn = (n_x)_{x\in Q_0}$ of nonnegative integers.
Given a family of vector spaces~$\CV=(V_x)_{x\in Q_0}$, denote by $\dim\CV$ the dimension vector with components~$(\dim \CV)_x = \dim V_x$.
The space of \emph{representations} of the quiver~$Q$ on~$\CV$ is given by
\begin{align*}
  \CH_Q(\CV) \coloneqq \bigoplus_{a : x\to y\in Q_1} \Hom(V_x, V_y).
\end{align*}
Its elements are families~$r=(r_a)_{a\in Q_1}$ of linear maps~$r_a : V_x\to V_y$, one for each arrow $a : x\to y$ in $Q_1$.
The Lie group~$\GL_Q(\CV) = \prod_{x\in Q_0} \GL(V_x)$ acts on $\CH_Q(\CV)$ as follows:
given any $g\in \GL_Q(\CV)$ and $r\in \CH_Q(\CV)$, the result of acting by $g$ on $r$ is denoted by~$grg^{-1}$, where
 \begin{align*}
   (g r g^{-1})_{a : x\to y\in Q_1} = g_y r_a g_x^{-1}.
 \end{align*}

Let $\CS=(S_x)_{x\in Q_0}$ be a family of subspaces~$S_x \subseteq V_x$.
Then~$\CS$ is called a \emph{subrepresentation} of~$r\in\CH_Q(\CV)$ if $r_a (S_x)\subset S_y$ for every arrow~$a: x\to y$ in~$Q_1$; we denote this by $r(\CS)\subseteq \CS$.
% That is, $r$ restricts to a representation in~$\CH_Q(\CS)$.
That is, there exists a representation~$s \in \CH_Q(\CS)$ such that $\iota_y \circ s_a = r_a \circ \iota_x$ for every arrow $a: x\to y$ in~$Q_1$, where $\iota_x : S_x \subseteq V_x$ denotes the inclusion map.
Similarly, a \emph{quotient representation} of $r\in \CH_Q(\CV)$ is given by a family of vector spaces~$\CW=(W_x)_{x \in Q_0}$ and surjective linear maps~$q_x:V_x\to W_x$ such that there exists $\bar r \in \CH_Q(\CW)$ with $q_y \circ r_a = \bar r_a \circ q_x$ for every arrow~$a: x\to y$. % in~$Q_1$.
Equivalently, $\ker(q) \coloneqq (\ker(q_x))_{x \in Q_0}$ is a subrepresentation of~$r$.
Denote by~$Q^*$ the quiver $(Q_0,Q_1^*)$, which has the same vertices as~$Q$ but the directions of all arrows are reversed.
If $r\in \CH_Q(\CV)$, then the adjoint~$r^*$ is a representation in $\CH_{Q^*}(\CV^*)$.
Then a quotient representation of $r\in \CH_Q(\CV)$ in~$\CW$ gives rise to a subrepresentation~$\CW^*$ of~$r^*$, and vice versa.

Schofield characterized inductively the dimension vectors~$\valpha$ such that every~$r\in\CH_Q(\CV)$ has a subrepresentation~$\CS$ with $\dim\CS = \valpha$.
(It is easy to see that it suffices to check this condition on a Zariski-open set of representations~$r$.)
We call such a dimension vector a \emph{Schofield subdimension vector} and denote this by $\valpha\hookrightarrow_Q \vn$, where~$\dim\CV = \vn$.
Similarly, we can consider the dimension vectors~$\vbeta$ such that any~$r\in\CH_Q(\CV)$ has a quotient representation in~$\CW$ with $\dim\CW = \vbeta$.
We call such a dimension vector a \emph{Schofield quotient dimension vector} and denote this property by $\vn \twoheadrightarrow_Q \vbeta$, where $\dim\CV = \vn$.
Given dimension vectors~$\valpha, \vbeta$ such that $\vn = \valpha + \vbeta$, the condition $\valpha \hookrightarrow_Q \vn$ is equivalent to~$\vn \twoheadrightarrow_Q \vbeta$, as well as to~$\vbeta \hookrightarrow_{Q^*} \vn$.

Given $\CV$ and a dimension vector $\valpha$, $\Gr_Q(\valpha,\CV) = \prod_{x\in Q_0} \Gr(\alpha_x, V_x)$ is defined as the product of the Grassmannians~$\Gr(\alpha_x,V_x)$ of subspaces of~$V_x$ of dimension~$\alpha_x$.
For a given representation~$r\in\CH_Q(\CV)$, we define the corresponding \emph{quiver Grassmannian} by
\begin{align*}
  \Gr_Q(\valpha,\CV)_r \coloneqq \{ \CS \in \Gr_Q(\valpha, \CV) : r\CS \subseteq \CS \}.
\end{align*}
Thus a Schofield subdimension vector is a dimension vector~$\valpha$ such that $\Gr_Q(\valpha,\CV)_r \neq \emptyset$ for every representation~$r\in\CH_Q(\CV)$.
Let $\vn = \dim \CV$.
Then the dimension of each irreducible component of the quiver Grassmannian~$\Gr_Q(\valpha,\CV)_r$ is, for generic $r\in\CH_Q(\CV)$, given by $\Euler_Q(\valpha,\vn - \valpha)$, where we have introduced the \emph{Euler bilinear form}
\begin{align}\label{eq:euler form}
  \Euler_Q(\valpha,\vbeta)
\coloneqq \sum_{x\in Q_0} \alpha_x \beta_x - \!\!\!\!\sum_{a:x\to y\in Q_1}\!\!\!\! \alpha_x\beta_y.
\end{align}
Thus, a necessary condition for $\valpha \hookrightarrow_Q \vn$ is $\Euler_Q(\valpha,\vn - \valpha)\geq 0$.
This is not sufficient.
Necessary and sufficient conditions are given by the following theorem by Schofield~\cite{MR1162487}*{Theorem~5.4}.

\begin{theorem}[Schofield]
Let ${\valpha},{\vbeta}$ be dimension vectors and~$\vn=\valpha+\vbeta$.
Then $ \valpha \hookrightarrow_Q \vn$  if and only if
$\Euler_Q(\vgamma,\vbeta)\geq 0$ for any $\vgamma \hookrightarrow_Q \valpha$.
\end{theorem}

Equivalently, using the quiver~$Q^*$ and the fact that $\Euler_Q(\valpha,\vbeta)=\Euler_{Q^*}(\vbeta,\valpha)$, necessary and sufficient conditions for ${\vbeta}$ to be a Schofield quotient dimension vector are given as follows:

\begin{theorem}[Schofield]\label{theo:crit}
Let ${\valpha},{\vbeta}$ be dimension vectors and $\vn=\valpha+\vbeta$.
Then $\vn\twoheadrightarrow_Q  \vbeta$  if and only if
$\Euler_Q(\valpha,\vgamma)\geq 0$ for every~$\vgamma$
such that~$\vbeta \twoheadrightarrow_Q \vgamma$ (equivalently, $\vgamma \hookrightarrow_{Q^*} \vbeta$).
\end{theorem}

We will use the latter theorem in the following.

%=============================================================================
\section{Subrepresentations and Semi-invariant Polynomials}
%=============================================================================

Let $\z_Q = \oplus_{x \in Q_0} \R \vdelta_x$ denote the vector space with basis~$\vdelta_x$ for~$x \in Q_0$, and let~$\z_Q^*=\oplus_{x\in Q_0} \R \vdelta^x$ denote its dual space, with dual basis~$\vdelta^x$.
The natural pairing of $\vlambda = \sum_x \lambda_x \vdelta^x \in \z_Q^*$ and $\vsigma = \sum_x \sigma_x \vdelta_x \in \z_Q$ is~given~by
\begin{align*}
  \braket{\vlambda, \vsigma} = \sum_{x \in Q_0} \lambda_x \sigma_x.
\end{align*}
Define the lattices $\Gamma = \oplus_{x \in Q_0} \Z \vdelta_x \subseteq \z_Q$ and $\Lambda = \oplus_{x \in Q_0} \Z \vdelta^x \subseteq \z_Q^*$.
Note that dimension vectors are elements of~$\Gamma$ with nonnegative entries.
The elements of $\Lambda$ are called \emph{weights}.
Given a weight $\vsigma \in \Lambda$, we let~$\CP_\vsigma(\CV)$ denote the space of semi-invariant polynomial functions~$p$ on $\CH_Q(\CV)$ of weight~$\vsigma$, meaning
\begin{equation}\label{eq:def semi}
p(g r g^{-1})=\prod_{x\in Q_0} \det(g_x)^{-\sigma_x} p(r)
\qquad (\forall g\in \GL_Q(\CV), r\in \CH_Q(\CV)).
\end{equation}
Our convention differs by a sign from the definition in~\cite{MR1758750}, in order to be consistent with the general definition of an isotypical subspace.
The following lemma was proved by King~\cite{MR1315461}.

\begin{definition}\label{def:Kingcone}
Define $\Sigma_{Q,\geo}(\vn) \subset \z_Q^*$ as the rational polyhedral cone consisting of all elements~$\vsigma\in \z_Q^*$ such that
\begin{align}\label{eq:lambda eq}
  \langle\vsigma,\vn\rangle = 0,
\end{align}
and
\begin{align}\label{eq:lambda ieq}
  \langle\vsigma, \vbeta\rangle \geq 0
\end{align}
for every dimension vector~$\vbeta$ such that $\vn\twoheadrightarrow_Q \vbeta$ (equivalently, $\vbeta\hookrightarrow_{Q^*}\vn$).
\end{definition}

\begin{lemma}[King]\label{lem:king}
Assume there exists a nonzero semi-invariant polynomial on $\CH_Q(\CV)$ of weight~$\vsigma$, and let $\vn = \dim \CV$.
Then, $\vsigma \in \Sigma_{Q,\geo}(\vn)$.
\end{lemma}

Define $\valpha \coloneqq \vn-\vbeta$.
Using~\eqref{eq:lambda eq}, we see that the inequalities~\eqref{eq:lambda ieq} are equivalent to
\begin{align}\label{eq:lambda ieq sub}
  \langle\vsigma, \valpha\rangle \leq 0
\end{align}
for all $\valpha$ such that $\valpha \hookrightarrow_Q \vn$ (which is equivalent to $\vn \twoheadrightarrow_Q \vbeta$).
We recall the beautiful proof of King's lemma.
\begin{proof}
Condition~\eqref{eq:lambda eq} follows by considering the one-parameter subgroup $(e^t  {\rm\Id}_{V_x})_{x\in Q_0}$ of $GL_Q(\CV)$, which acts trivially on $\CH_Q(\CV)$.
To establish~\eqref{eq:lambda ieq} we will prove the equivalent condition~\eqref{eq:lambda ieq sub}.
Fix families $\CS, \CW$ of vector spaces with $\dim\CS = \valpha$ and $\dim\CW=\vbeta$ such that~$\CV = \CS \op \CW$.
Let $\valpha\hookrightarrow_Q \vn$.
This means that, under the action of $GL_Q(\CV)$, any element of $\CH_Q(\CV)$ can be transformed to an element $r=(r_a)_{a\in Q_1}$ with $r_a$ of the form
\begin{equation}\label{eq:ra}
r_a =\begin{bmatrix}
       r_{11} & r_{12} \\
       0 & r_{22} \\
     \end{bmatrix},
  \end{equation}
with respect to $V_x=S_x\oplus W_x$ and $V_y=S_y\oplus W_y$.
We will say that such an  $r$ is of \emph{parabolic type~$\valpha$}.
Thus, if the semi-invariant polynomial~$p$ is nonzero, there is some $r =(r_a)_{a\in Q_1}$  of parabolic type $\valpha$ with~$p(r)\neq 0$.
For $t\in\R$, consider $g(t)=(g_x(t)) \in \GL_Q(\CV)$ given by
\[ g_x(t)=\begin{bmatrix}
       e^t & 0 \\
       0 & 1\\
\end{bmatrix} \]
with respect to $V_x = S_x \oplus W_x$.
Because $\det g_x(t)=e^{t\alpha_x}$, the semi-invariance condition~\eqref{eq:def semi} implies that
\[ p(g(t) r g(t)^{-1}) = e^{-t \langle \vsigma, \valpha \rangle} p(r). \]
On the other hand,
\[ g(t)r g(t)^{-1}= \begin{bmatrix}
       r_{11} & e^t r_{12} \\
       0 & r_{22} \\
     \end{bmatrix} \]
has a limit when $t\to -\infty$.
Thus, necessarily $\langle \vsigma, \valpha \rangle \leq 0$.
\end{proof}

The following observation is a special case of the important \cref{lem:main} below.

\begin{lemma}\label{lem:min}
Let $\vsigma \in \Sigma_{Q,\geo}(\vn)$.
Let $x \in Q_0$ be a vertex that has no incoming arrows.
If $n_x > 0$, then $\sigma_x \geq 0$.
\end{lemma}
\begin{proof}
  Let $\CV$ be a family of vector spaces of dimension vector~$\vn$.
  Pick~any subspace~$S_x \subseteq V_x$ of codimension~1, and let $S_y \coloneqq  V_y$ for~$y \neq x$.
  Since~$x$ has no incoming arrows, $\CS$ is a subrepresentation of any representation~$r \in \CH_Q(\CV)$.
  In other words, for~$\valpha \coloneqq \dim \CS$ it holds that~$\valpha \hookrightarrow_Q \vn$.
  Therefore $\vn \twoheadrightarrow_Q \vbeta \coloneqq \vn - \valpha = \vdelta_x$.
  Since~$\vsigma \in \Sigma_{Q,\geo}(\vn)$, we conclude that $\sigma_x = \langle\vsigma, \vdelta_x \rangle \geq 0$.
\end{proof}

We recall the well-known construction of semi-invariant polynomials on~$\CH_Q(\CV)$.
Since the quiver~$Q$ has no cycles, there are isomorphisms~$L_1,L_2: \z_Q\to \z_Q^*$ such that
for every $\vxi\in \z_Q$ and $\valpha,\vbeta\in \z_Q$:
\begin{align*}
  \Euler_Q(\valpha,\vxi)&=\langle L_1(\valpha),\vxi \rangle, \\
  \Euler_Q(\vxi,\vbeta) &=\langle L_2(\vbeta),\vxi \rangle.
\end{align*}
The maps $L_1,L_2$ define isomorphisms between the lattices~$\Gamma \subset \z_Q$ and~$\Lambda \subset \z_Q^*$.
In particular, if $\valpha,\vbeta$ are dimension vectors, then $L_1(\valpha), L_2(\vbeta)$ are in the weight lattice~$\Lambda$.

We now change notation since in order to construct semi-invariant polynomials on $\CH_Q(\CV)$
we will need to realize $\CV$ as a quotient of a larger family of vector spaces $\CE$.
 % $q:\CE\to \CV$.
Let ${\vA}$, ${\vB}$ be two dimension vectors.
Pick~$\CF=(F_x)_{x\in Q_0}$ and $\CG=(G_x)_{x\in Q_0}$
such that $\dim \CF={\vA}$ and $\dim \CG={\vB}$, and define
\begin{align*}
 \CH_Q(\CF,\CG) &\coloneqq \!\!\!\!\!\bigoplus_{a:x\to y \in Q_1}\!\!\!\!\! \Hom(F_x,G_y), \\
  \g_Q(\CF,\CG) &\coloneqq \bigoplus_{x\in Q_0} \Hom(F_x,G_x).
\end{align*}
Then, $\Euler_Q({\vA},{\vB}) = \dim \g_Q(\CF,\CG)-\dim\CH_Q(\CF,\CG)$.
Thus the two spaces have the same dimension if~$\Euler_Q({\vA},{\vB})=0$.
For~$v\in \CH_Q(\CF)$ and~$w\in \CH_Q(\CG)$, consider the map
\begin{align*}
  \delta_{v,w}\colon \g_{Q}(\CF,\CG)\to \CH_Q(\CF,\CG), \quad \Phi \mapsto \Phi v - w\Phi,
\end{align*}
where $\Phi v - w\Phi$ is the element of $\CH_Q(\CF,\CG)$ given by
$(\Phi v - w\Phi)_{a : x \to y \in Q_1} =
  \Phi_y v_a - w_a \Phi_x %\colon F_x \to G_y
$.
Finally, let $\CE=(E_x)_{x\in Q_0}$ be such that $\dim \CE=\vA+\vB$.
Then we have the following lemma due to Schofield~\cite{MR1162487}.

\begin{lemma}[Schofield]\label{lem:quo}
Let ${\vA},{\vB}$ be dimension vectors and~$\vN={\vA}+{\vB}$.
Assume $\Euler_Q(\vA,\vB)=0$.
Then, ${\vA}\hookrightarrow_Q \vN$ (equivalently, ${\vN}\twoheadrightarrow_Q \vB$)
if and only if
the map $\delta_{v,w}$ is an isomorphism for generic~$(v,w)$ in $\CH_Q(\CF)\oplus \CH_Q(\CG)$.
\end{lemma}

The lemma follows easily from the fact that the map~$(g,r)\mapsto g r g^{-1}$, where $g \in \GL_Q(\CE)$ and $r$ varies in the set of representations of parabolic type~$\vA$ (a notion defined in the proof of \cref{lem:king}), surjects onto~$\CH_Q(\CE)$ if (and only if) $\vA\hookrightarrow_Q \vN$.

\begin{definition}[Determinantal semi-invariants]\label{def:cvw}
Let ${\vA},{\vB}$ be dimension vectors with
$\Euler_Q(\vA,\vB)=0$.
Let~$\CF,\CG$  with $\dim \CF=\vA$ and $\dim \CG=\vB$.
 % Let $\vN={\vA}+{\vB}$.
Then we can define the polynomial function
\begin{align*}
  C^{{\vA},{\vB}} \colon \CH_Q(\CF)\oplus \CH_Q(\CG) \to \C, \qquad
  C^{{\vA},{\vB}}(v,w) \coloneqq \det(\delta_{v,w}),
\end{align*}
for an arbitrary choice of bases of the spaces $\g_{Q}(\CF,\CG)$ and $\CH_Q(\CF,\CG)$.
Different choices of bases only change the function by a nonzero scalar multiple.
\end{definition}

The following proposition is due to Schofield-Van den Bergh~\cite{MR1908144}, and Derksen-Weyman~\cite{MR1758750}.

\begin{proposition}\label{pro:deter}
For all $g\in \GL_Q(\vA),$ $h\in \GL_Q(\vB)$,
$v\in \CH_Q(\CF)$, and $w\in \CH_Q(\CG)$, one has
\[ C^{{\vA},{\vB}}(gvg^{-1},h w h^{-1})=\det(g)^{L_2(\vB)} \det(h)^{-L_1(\vA)} C^{{\vA},{\vB}}(v, w). \]
\end{proposition}

By \cref{lem:quo}, $C^{{\vA},{\vB}}$ is a non-zero polynomial if and only if $\vA\hookrightarrow_Q \vN$ or, equivalently, $\vN\twoheadrightarrow_Q\vB$.

\begin{corollary}\label{cor:deter}
Let ${\vA},{\vB}$ be dimension vectors and~$\vN={\vA}+{\vB}$.
Assume that~$\Euler_Q(\vA,\vB)=0$.
Pick $\CF$ and~$\CG$  with $\dim \CF=\vA$, $\dim \CG=\vB$.
Then, ${\vA}\hookrightarrow_Q \vN$ (equivalently, ${\vN}\twoheadrightarrow_Q \vB$)
if and only if $C^{{\vA},{\vB}}(v,\cdot)$ is a nonzero semi-invariant of weight $\vsigma=L_1(\vA)$ on $\CH_Q(\CG)$ for generic~$v\in\CH_Q(\CF)$.
% if and only if, for generic $w \in \CH_Q(\CG)$, $C^{{\vA},{\vB}}(\cdot,w)$ is a semi-invariant of weight $\vsigma=-L_2(\vB)$ on $\CH_Q(\CF)$,
\end{corollary}

%=============================================================================
\section{Derksen-Weyman saturation theorem}
%=============================================================================
In this section, we give our proof of the saturation theorem of Derksen-Weyman~\cite{MR1758750}.
The theorem states that the following semigroup is satured and gives a concrete description in terms of linear inequalities.

\begin{definition}
Let $\CV$ be a family of vector spaces with dimension vector~$\vn$.
Define $\Sigma_{\rep,Q}(\vn) \subset \Lambda$ as the semigroup consisting of the weights $\vsigma$ such that there exists a non-zero semi-invariant polynomial of weight $\vsigma$ on $\CH_Q(\CV)$.
\end{definition}

\begin{theorem}\label{theo:sat}
$\Sigma_{\rep,Q}(\vn)=\Sigma_{\geo,Q}(\vn)\cap \Lambda$.
\end{theorem}

To prove the theorem, we can assume without loss of generality that $n_x > 0$ for all $x \in Q_0$.
Indeed, if we denote by $Q'$ the quiver obtained by removing all vertices~$x$ with~$n_x = 0$, and by $\vn', \vsigma'$ the dimension vector and weight with all corresponding entries removed, then it is immediate that $\vsigma \in \Sigma_{\rep,Q}(\vn)$ if and only if $\vsigma' \in \Sigma_{\rep,Q'}(\vn')$, and likewise for $\Sigma_{\geo,Q}(\vn)$.

The main ingredient in our proof is the observation that if a weight~$\vsigma$ satisfies King's inequalities for~$\vn$ (\cref{def:Kingcone}), $\vA \coloneqq L_1^{-1}(\vsigma)$ is a dimension vector that satisfies Schofield's inequalities for~$\vA + \vn \twoheadrightarrow_Q \vn$ (\cref{theo:crit}).
Recall that $\vA = L_1^{-1}(\vsigma)$ means that $\Euler_Q(\vA,\vxi) = \langle \vsigma,\vxi \rangle$ for all $\vxi\in \z_Q$.

\begin{proposition}\label{pro:king vs schofield}
Let $\vsigma\in \Sigma_{Q,\geo}(\vn) \cap \Lambda$.
Then, $\vA \coloneqq L_1^{-1}(\vsigma)$ is a dimension vector and it holds that~$\Euler_Q(\vA, \vgamma) \geq 0$ for every dimension vector~$\vgamma$ such that $\vn \twoheadrightarrow_Q \vgamma$.
Moreover, it holds that~$\Euler_Q(\vA, \vn) = 0$.
\end{proposition}
\begin{proof}
In \cref{lem:main} below, we prove that~$\vA$ is a dimension vector.
Now the remaining claims follows immediately from $\Euler_Q(\vA,\cdot) = \langle \vsigma, \cdot \rangle$ and the defining inequalities and equation of $\Sigma_{Q,\geo}(\vn)$.
% it is easy to deduce the claim from the inequalities and the equation defining~$\Sigma_{Q,\geo}(\vn)$ (\cref{def:Kingcone}): we have
% \begin{equation*}
%   \Euler_Q(\vA, \vn) = \langle \vsigma, \vn \rangle = 0
% \end{equation*}
% and, for every~$\vgamma$ such that~$\vn \twoheadrightarrow_Q \vgamma$,
% \begin{equation*}
%   \Euler_Q(\vA, \vgamma) = \langle \vsigma, \vgamma \rangle \geq 0.
%   \qedhere
% \end{equation*}
\end{proof}

It remains to prove the central lemma used in the above proof.

\begin{lemma}\label{lem:main}
If $\vsigma\in \Sigma_{Q,\geo}(\vn) \cap \Lambda$, then $\vA \coloneqq L_1^{-1}(\vsigma)$ is a dimension vector.
That is, $\vA \in \Gamma$ and $A_x \geq 0$ for all $x \in Q_0$.
\end{lemma}

\begin{proof}
Because~$L_1$ is an isomorphism between~$\Gamma$ and~$\Lambda$, we already know that~$\vA \in \Gamma$, so it remains to prove that it is entrywise nonnegative.
Because~$Q$ is acyclic, we can enumerate the vertices as~$Q_0 = \{x_1,\dots,x_K\}$ in a topological order.
% \MW{I would almost find it easier to just call the vertices $1,\dots,K$, and refer to all matrix entries as $M_{ij}$... but note that $m_{ij} = M_{ji}$.}
That is, if $m_{ij}$ denotes the number of arrows from $x_i$ to $x_j$ then we have $m_{ij} = 0$ if $i > j$.
Define the map
\[ M \colon \z_Q \to \z_Q, \quad M(\vdelta_{x_i}) = \sum_{i < j} m_{ij} \vdelta_{x_j} = \sum_{a : x_i \to x_j \in Q_1} \vdelta_{x_j}. \]
Let $\Iso \colon \z_Q^*\to \z_Q$ denote the map such that $\Iso(\vdelta^x)=\vdelta_x$ for all $x \in Q_0$.
Then we have
\[ \Euler_Q(\valpha, \Iso(\vtau)) = \langle \vtau, (I - M) \valpha \rangle
\qquad (\forall \valpha \in \z_Q, \vtau \in \z_Q^*)
\]
Observe that $M$ is strictly lower triangular.
Hence $I - M$ is invertible and its inverse is again an integer matrix:
\begin{align*}
  P \coloneqq (I - M)^{-1} = \sum_{k=0}^K M^k.
\end{align*}
Observe that $P\vdelta_{x_i} =\sum_j  p_{ij}\vdelta_{x_j}$, with $p_{ij}$ the number of paths from~$x_i$ to~$x_j$.
By convention, the number of paths of length~$0$ is $1$, so $p_{ii}=1$.
We have $p_{ij}=0$ if~$i>j$.
In particular,
\begin{align}\label{eq:A}
  \vA = (I - M)^{-1} \Iso(\vsigma) = P \Iso(\vsigma),
\end{align}
that is,
\begin{align*}
  A_{x_k} = \sigma_{x_k} + \sum_{i < k} p_{ik} \sigma_{x_i}.
\end{align*}
% We need to prove that $A_{x_k}\geq 0$ for every $k$.%
To prove that ${A}_{x_k}\geq 0$, we will use the inequalities $\langle \vsigma,\vbeta\rangle \geq 0$ for natural Schofield subdimension vectors  $\vbeta\hookrightarrow_{Q^*}\vn$.%
\footnote{Note that since $x_1$ is minimal, we have $A_{x_1}=\sigma_{x_1} \geq 0$ by \cref{lem:min}.}
To this end we consider subrepresentations generated by nonzero vectors $0\neq v_k\in \CV^*_{x_k}$ (such vectors exist since $n_{x_k} > 0$).
Let $s=(s_a)_{a\in Q_1^*}\in \CH_{Q^*}(\CV^*)$.
If~$j<k$ and~$p$ is any path from~$x_k$ to~$x_j$ in the quiver~$Q^*$, we denote by~$s_p\in \Hom(\CV^*_{x_k},\CV^*_{x_j})$
the composition of the matrices~$s_a$ along the path~$p$.
Denote by $\Sub(v_k,s,x_j)$ the subspace of~$\CV_{x_j}^*$ generated by the elements~$s_p(v_k)$ of $\CV_{x_{j}}^*$, where $p$ runs on all paths from~$x_k$ to~$x_j$ in~$Q^*$.
Because in the quiver~$Q^*$ we have $p_{jk}$ paths from~$x_k $ to~$x_j$, the dimension of this subspace is less or equal to~$p_{jk}$.%
\footnote{If all entries of the dimension vector $\vn$  are sufficiently large, then the dimension of this subspace is equal to $p_{jk}$ for generic $s$.}
For $j=k$, define  $\Sub(v_k,s,x_j)=\C v_k$.
For $j>k$, define  $\Sub(v_k,s,x_j)=0$.
By construction, for any $v_k\in \CV_{x_k}^*$, the family $\CS(v_k,s)$
of subspaces $\CS(v_k,s)_{x_j}=\Sub(v_k,s,x_j)$ defines a subrepresentation of~$s$.
Let $t_{jk}$ denote the maximal dimension of $\Sub(v_k,s,x_j)$ when $s$ varies in $\CH_{Q^*}(\CV^*)$.
This maximum does not depend the choice of nonzero~$v_k$ in $\CV_{x_k}^*$, since we can conjugate $v_k$ to any other non-zero element by the action of $\GL_Q(\CV^*)$.
Thus, for any choice of~$v_k$ there is a Zariski open subset of~$s \in \CH_{Q^*}(\CV^*)$ such that $\dim \CS(v_k,s) = \vtheta_k$, where $\vtheta_k \coloneqq \sum_j t_{jk} \vdelta_{x_j}$.
It follows that~$\vtheta_k \hookrightarrow_{Q^*} \vn$, i.e., $\vn \twoheadrightarrow_Q \vtheta_k$, because $\vsigma \in \Sigma_{Q,\geo}(\vn)$ we see that%
\footnote{For $\vn$ large, $ \langle \vsigma, \vtheta_k\rangle =\sigma_k +\sum_{i<k} p_{ik} \sigma_i$, so is precisely equal to $A_{x_k}$, and hence $A_{x_k} \geq 0$.
So, at least when $\vn$ is large, \cref{lem:main} is clear.}
\[ \langle \vsigma, \vtheta_k \rangle \geq 0. \]
Define
\[ T\colon\z_Q \to \z_Q, \quad T(\vdelta_{x_i})=\sum_{i\leq j} t_{ij} \vdelta_{x_j}, \]
and consider
\begin{align*}
  W \coloneqq I - T(I - M) = (I - T) + TM
\end{align*}
We will now show that $W$ is entrywise nonnegative.%
\footnote{For $\vn$ large, $T=P=(I-M)^{-1}$, so $W=0$.}
The matrix~$W$ is strictly lower triangular since~$M$ and~$(I-T)$ are strictly lower triangular.
In fact:
\begin{align*}
  W(\vdelta_{x_i})
= -\sum_{i<k}t_{ik} \vdelta_{x_k}+\sum_{i<j\leq k}m_{ij}t_{jk} \vdelta_{x_k}
\end{align*}
Thus, the matrix~$W$ has nonnegative entries if and only if, for all $i < k$,
\begin{equation}\label{eq:nonneg}
  t_{ik} \leq \sum_{i<j\leq k} m_{ij} t_{jk} = m_{ik} + \sum_{i<j<k} m_{ij} t_{jk}.
\end{equation}
Recall that, for $s$ in a Zariski open subset of $\CH_{Q^*}(\CV^*)$, for all $k > i$, $t_{ik}$ is the dimension of the vector space $ \Sub(v_k,s,x_i) \subseteq \CV_{x_i}^*$  spanned by the vectors $s_p(v_k)$ where $p$ runs over all paths from~$x_k$ to~$x_i$ in $Q^*$.
There are $m_{ik}$ arrows from~$x_k$ to~$x_i$, hence the paths of length one contribute at most~$m_{ik}$ to~$t_{ik}$.
All other paths are composed of an arrow from~$x_k$ to some~$x_j$ and some path from~$x_j$ to~$x_k$, hence they contribute at most~$m_{ij} t_{jk}$, for any~$j$.
This proves \cref{eq:nonneg} and hence all entries of the matrix~$W$ are nonnegative.
% The paths going directly from $k$ to $i$ in one step are in numbers $m_{ik}$.
% So the  subspace of $L(i,k)$ obtained by  considering these one step paths is of dimension less or equal to $m_{ik}$.
% Otherwise we have at least two steps,
% and  paths are composed  of paths $q$ from  $x_k$ to $x_j$
% and paths $a:x_j$ to $x_i$  in numbers $m_{ij}$.
% The subspace of  $L(i,k)$ spanned by the vectors $s_a s_q v_k$ is thus of dimension less that
% $m_{i,j} t_{jk}$.

Now recall from \cref{eq:A} that $\vA = (I - M)^{-1} \Iso(\vsigma)$.
It follows that $W \vA = \vA - T \Iso(\vsigma)$ and hence
\begin{align*}
  \vA = W \vA + T \Iso(\vsigma) = W \vA + \sum_k \langle \vsigma, \vtheta_k \rangle \vdelta_{x_k},
\end{align*}
where $W$~is strictly lower triangular with nonnegative entries and $\langle \vsigma, \vtheta_k \rangle \geq 0$ for all~$k$.
Thus we must have $A_1 = \langle \vsigma, \vtheta_1 \rangle \geq 0$,%
\footnote{We note that this is equivalent to, and hence reproves, \cref{lem:min}.} and then it follows by induction that~$A_{x_k} \geq 0$ for all~$k$.
This concludes the proof.
\end{proof}

\begin{example}
Consider a quiver with $K=3$ vertices and
\[ M = \begin{bmatrix} 0&0&0\\ m_{{1,2}}&0&0
\\ m_{{1,3}}&m_{{2,3}}&0\end{bmatrix}. \]
Then,
\begin{align*}
  P&=\begin{bmatrix} 1&0&0\\ m_{{1,2}}&1&0
\\ m_{{2,3}}m_{{1,2}}+m_{{1,3}}&m_{{2,3}}&1
\end{bmatrix},
\qquad
  T=\begin{bmatrix} 1&0&0\\ t_{{1,2}}&1&0
\\ t_{{1,3}}&t_{{2,3}}&1
\end{bmatrix},
\intertext{and}
W&=\begin{bmatrix} 0&0&0\\ -t_{{1,2}}+m_{{1
,2}}&0&0\\ t_{{2,3}}m_{{1,2}}+m_{{1,3}}-t_{{1,3}}&-t
_{{2,3}}+m_{{2,3}}&0\end{bmatrix}.
\end{align*}
\end{example}

\begin{proof}[Proof of \cref{theo:sat}]
The inclusion $\subseteq$ holds due to King's lemma (\cref{lem:king}).

For the converse inclusion $\supseteq$, let $\vsigma\in \Sigma_{Q,\geo}(\vn) \cap \Lambda$ and set $\vA \coloneqq L_1^{-1}(\vsigma)$.
\Cref{pro:king vs schofield}, along with Schofield's criterion (\cref{theo:crit}), implies that~$\vA + \vn \twoheadrightarrow \vn$.
The proposition also states that $\Euler_Q(\vA, \vn) = 0$.
Using \cref{cor:deter} (with~$\vA$, $\vN = \vA + \vn$, and $\vB = \vn$), we conclude that $C^{\vA,\vn}(v,\cdot)$ is a nonzero semi-invariant of weight~$\vsigma$ on~$\CH_Q(\CG)$ for generic $v \in \CH_Q(\CF)$, where $\dim\CF = \vA$ and $\dim\CG = \dim\CV = \vn$.
Thus, $\vsigma \in \Sigma_{\rep,Q}(\vn)$.
\end{proof}

%=============================================================================
\section*{Acknowledgments}
%=============================================================================
We thank Harm Derksen for sharing his insights on the history of different approaches to the saturation theorem.

%=============================================================================

%=============================================================================
\end{document}